\documentclass[10pt]{article}
\usepackage{amsmath}
\usepackage{amsfonts}
\usepackage{amssymb}
\usepackage{amsthm}
\usepackage{hyperref}
\usepackage[alwaysadjust]{paralist}
\setdefaultenum{\footnotesize (1.)}{\footnotesize (a)}{\footnotesize (i.)}{\footnotesize A.}
\setdefaultitem{\textbullet}{$\circ$}{\small *}{ \texthyphen }
\usepackage{xcolor}

\title{On rapid idempotent ultrafilters}
\author{Peter Krautzberger\thanks{Partially supported by DFG-grant KR 3818; Subject classification 03E75 (Primary) 54D80, 05D10 (Secondary)} }

\newtheorem{Proposition}{Proposition}
\newtheorem{Theorem}{Theorem}
\newtheorem{Definition}{Definition}

\begin{document}
\maketitle

This short note contains the proofs of two small but somewhat surprising results about ultrafilters on $\mathbb{N}$: 1. strongly summable ultrafilters are rapid, 2. every rapid ultrafilter induces a closed left ideal of rapid ultrafilters. As a consequence, there will be rapid minimal idempotents in all models of set theory with rapid ultrafilters.

The history of this result has been published as an experiment in mathematical writing on the author's website \cite{Krautzberger2012}, \cite{Krautzberger20122} where you can can also find additional remarks by Andreas R. Blass and Neil Hindman, offering a form of peer-review.

\section{Preliminaries}

We assume basic knowledge about ultrafilters and the algebra in the Stone-\v{C}ech compactification; the notation and conventions follow the standard monograph on the topic, \cite{HindmanStrauss}.

We work in the semigroup $(\mathbb{N},+)$ though it should be straight forward to extend the two results to countable semigroups with finite-to-one multiplication maps in general.

\begin{Definition}
\begin{itemize}
 \item An \emph{FS-set} (in $\mathbb{N}$) is a set of the form 

\[FS(x_n) = \{ \sum_{n \in s} x_n : s\in [\omega]^{<\omega}, s\neq \emptyset \} \] where $(x_n)_{n\in \omega}$ is a sequence in $\mathbb{N}$.

\item A \emph{strongly summable ultrafilter} is an ultrafilter on $\mathbb{N}$ with a base of FS-sets.

\item An ultrafilter $p$ on $\mathbb{N}$ is \emph{rapid} if for every $f:\mathbb{N} \to \mathbb{N}$, there exists $A\in p$ such that $|A \cap f(i) | < i|$ for all $i$; in other words, the natural enumeration of $A$ dominates $f$ in the usual sense of functions. We simply say that $A$ \emph{dominates} $f$.

\end{itemize}

\end{Definition}

For a detailed introduction to strongly summable ultrafilters see \cite[chapter 12]{HindmanStrauss} or \cite{Krautzberger10}. For rapid ultrafilters, see \cite{Miller1980}. It is easy to check that it suffices to consider strictly monotone $f$ to verify rapidity.

\section{Strongly summable ultrafilters are rapid.}

We need a two classical results regarding strongly summable ultrafilters. First, a result by Andreas Blass and Neil Hindman (with probably much older roots).  One of the simplest yet very fruitful ideas for studying FS-sets is to look at the minimum and maximum functions mapping $\sum_{i\in s} x_i$ to $x_{\min(s)}$ and $x_{\max(s)}$ respectively. However, this may not always be well defined since there might be many ways of writing an element in an FS-set as a sum of the generators (just consider $FS((n)_{n\in \mathbb{N}})$). The following results eliminates this concern for strongly summable ultrafilters.

\begin{Theorem}[Blass, Hindman \cite{BlassHindman87}]
If $p$ is strongly summable, then $p$ has a base of FS-sets with the following properties. If $FS(x_n)$ is a base set, then adddition ``behaves like disjoint union'', i.e.,

\[ \sum_{n \in s} x_n + \sum_{n\in t} x_n = \sum_{n \in v} x_n \Leftrightarrow s \cap t= \emptyset, s\cup t = v  \]

\[ \mbox{and } \sum_{n \in s} x_n = \sum_{n\in t} x_n \Leftrightarrow s = t .  \]

In particular, the map 

\[ \operatorname{x_n-max} :FS(x_n) \to \omega, \sum_{n \in s} x_n \mapsto x_{\max(s)}\]

is well defined.
\end{Theorem}

In other words, addition restricted to these FS-sets works essentially like taking ``disjoint union'' of the indices. 

Note that the above properties are hereditary, i.e., if $FS(x_n)$ has the above property, and $FS(y_k) \subseteq FS(x_n)$, then $FS(y_k)$ has this property (because the $y_k$ must have ``disjoint support'' since all their sums  are in $FS(x_n)$) -- so (as is often the case in ultrafilter arguments) we could have asked for just one such set to be in the strongly summable ultrafilter.

The second classic result is the following due to Pierre Matet, cf. \cite[Proposition 7.5]{Matet88}.

\begin{Theorem}[Matet]
If $p$ is a strongly summable ultrafilter, then there exists $FS(x_n)\in p$ (as in the above theorem) such that $\operatorname{x_n-max}(p)$ is rapid.
\end{Theorem}

Since Matet uses a somewhat unusual notation, it's probably easier to read the result in \cite[Theorem 2]{BlassHindman87} (where it is not mentioned but the proof yields it) or in \cite[Proposition 8]{Krautzberger10}. But we are ready to prove the first result.

\begin{Theorem}
If $p$ is strongly summable, then $p$ is rapid.
\end{Theorem}

The idea is simple: we know we can dominate any function in the min-image. But of course the set that witnesses this in the pre-image might be much larger. However, we know what these sets look like (FS-sets!) and how long it takes for a new image point to appear in the image ($2^n$-steps for $n$-steps in the image). So all we have to do is speed up functions by $2^n$ and let that sped-up function be dominated in the rapid $\operatorname{x_n-max}$-image. Then we pick an FS-set in our union ultrafilter that witnesses this domination, in particular, it's generating sequence will dominate that sped-up function. Finally, the FS-set will still grow fast enough to dominate the original function.

\begin{proof}
\begin{itemize}

\item By Matet's theorem pick $FS(x_n) \in p$ such that $\operatorname{x_n-max}(p)$ is rapid.
\item Now pick any $f: \mathbb{N} \rightarrow \mathbb{N}$; we may assume that $f$ is strictly monotone.
\item By Matet's theorem we can find a set $A \in \operatorname{x_n-max}(p)$ that dominates $f \circ 2^{n +1} $.
\item Now fix $ FS(y_k) \subseteq FS(x_n), FS(y_k) \in p $ such that $\operatorname{x_n-max}[FS(y_k)] \subseteq A$.
\item Then $FS(y_k)$ dominates $f$.
\begin{itemize}
     \item Let $i\in \mathbb{N}$. We'll show that $|FS(y_k) \cap f(i)| < i$.
     \item Pick the maximal $y_j < f(i)$. 
     \item Then $f(i) \cap FS(y_k) \subseteq f(i) \cap FS(y_0,\ldots, y_j)$, i.e., we only need to find out how large $j$ is.
     \item The set $FS(y_0,\ldots, y_j)$ contains exactly  $ 2^{j+1}$-many elements
     \item So to show that it contains less than $i$-many elements, we need to prove that $ i > 2^{j+1}$.
     \begin{itemize}
     \item Of course, by choice of $j$, $f(i) > y_j$ and generally $y_j \geq \operatorname{x_n-max}(y_j)$.
     \item Observe that $\operatorname{x_n-max}[FS(y_k)] = \operatorname{x_n-max}[\{ y_k: k \in \omega\} ] \subseteq A$ (since we have "disjoint sums").
     \item Therefore any $\operatorname{x_n-max}(y_k)$ is greater or equal to the $k$-th element of $A$.
     \item Since the $A$ dominates $f\circ (2^{n +1})$, the $k$-th element of $A$ is greater than $f(2^{k+1})$. In particular, $\operatorname{x_n-max}(y_j) > f(2^{j+1})$.
     \item In total we have, $f(i) > y_j \geq \operatorname{x_n-max}(y_j) > f(2^{j+1})$.
     \item By $f$'s monotonicity, $ i > 2^{j+1}$. 
     \end{itemize}
\end{itemize}
\end{itemize}
\end{proof}

\section{Rapid ultrafilters produce rapid idempotents}

Are there are other rapid idempotents besides strongly summable ultrafilters? It turns out we can get a very strong, positive answer to this question. 

For completeness, two standard definitions.

\begin{Definition}
Let $p,q \in \beta \mathbb{N}$.
\begin{itemize}
 \item $p\otimes q$ is the ultrafilter on $\mathbb{N} \times \mathbb{N}$ generated by the sets 
\[ \{ \bigcup_{v\in V} \{ v\} \times W_v : V \in p, (W_v)_{v\in V} \mbox{ in } q\}. \]

\item $p+q$ is defined as $+ (p \otimes q)$, i.e., the image of $p\otimes q$ under the map $+$.

\end{itemize}

\end{Definition}

We get our result almost trivially from two well-known results on rapid ultrafilters.

\begin{Proposition}
For any two ultrafilters $p, q \in \beta \mathbb{N}$.
\begin{itemize}
 \item \cite[Theorem 4]{Miller1980} If $p$ is rapid, then $q\otimes p$ is rapid.
 \item (Folklore) The finite-to-one image of a rapid ultrafilter is rapid.
\end{itemize}
\end{Proposition}

With this in hand, we can make an easy observation.

\begin{Proposition}
If $p$ is rapid, $q$ any ultrafilter, then $q+p$ is rapid.
\end{Proposition}

\begin{proof}
\begin{itemize}
\item Since $p$ is rapid, $q\otimes p$ is rapid.
\item Then $q+p = +(q\otimes p)$ and addition is a finite-to-one map.
\item Hence $q+p$ is rapid.
\end{itemize}
\end{proof}

Then we can state our theorem as a trivial corollary. For the definition of minimal idempotent, see \cite[Chapter 1]{HindmanStrauss}.

\begin{Theorem}
If there exists a rapid ultrafilter $p$, then there exist rapid idempotent ultrafilters. More precisly, then $\beta \mathbb{N} + p$  contains only rapid ultrafilters, in particular minimal idempotents which are rapid.
\end{Theorem}

As an afterthought, let us go back to strongly summable ultrafilters. By an old result, cf. \cite[Theorem 12.39]{HindmanStrauss}, strongly summable ultrafilters are right-maximal idempotents (in the usual order $p\leq_R q \Leftrightarrow p+q =p$). So in a model of set theory with strongly summable ultrafilters, we find the "full spectrum" of rapid idempotents in that partial order -- from (strongly) right maximal all the way to (right) minimal.


\begin{thebibliography}{Mat88}

\bibitem[BH87]{BlassHindman87}
Andreas~R. Blass and Neil Hindman.
\newblock {On strongly summable ultrafilters and union ultrafilters.}
\newblock {\em Trans. Am. Math. Soc.}, 304:83--99, 1987.

\bibitem[HS98]{HindmanStrauss}
Neil Hindman and Dona Strauss.
\newblock {\em {Algebra in the Stone-\v{C}ech compactification: theory and
  applications}}, volume~27.
\newblock de Gruyter, 1998.

\bibitem[Kra10]{Krautzberger10}
Peter Krautzberger.
\newblock {\em {On Strongly Summable Ultrafilters}}.
\newblock {New York J. Math.}, 16, 2010 

\bibitem[Kra121]{Krautzberger2012}
Peter Krautzberger.
\newblock {One day in Colorado or Strongly summable ultrafilters are rapid}, \url{http://boolesrings.org/krautzberger/2012/04/02/one-day-in-colorado-or-strongly-summable-ultrafilters-are-rapid/} 2012.

\bibitem[Kra122]{Krautzberger20122}
Peter Krautzberger.
\newblock {Rapid idempotent ultrafilters}, \url{http://boolesrings.org/krautzberger/2012/04/08/rapid-idempotent-ultrafilters/} 2012.

\bibitem[Mat88]{Matet88}
Pierre Matet.
\newblock {Some filters of partitions}.
\newblock {\em J. Symbolic Logic}, 53(2):540--553, 1988.

\bibitem[Mil80]{Miller1980}
Arnold~W Miller.
\newblock {There are Q no -points in Laver's model for the Borel conjecture}.
\newblock {\em Proc. Amer. Math. Soc.}, 78(1):103 --106, 1980.

\end{thebibliography}


\end{document}